\documentclass[12pt]{article} \pdfoutput=1

\usepackage{amsmath,amssymb,amsthm,graphicx,color,hyperref}

\numberwithin{equation}{section}
\newtheorem{theorem}{Theorem}[section]
\newtheorem{lemma}[theorem]{Lemma}
\newtheorem{proposition}[theorem]{Proposition}
\newtheorem{corollary}[theorem]{Corollary}
\makeatletter
\def\@MRExtract#1 #2!{#1}     % thanks, Martin!
\newcommand{\MR}[1]{% we need to strip the "(...)"
  \xdef\@MRSTRIP{\@MRExtract#1 !}%
  \href{http://www.ams.org/mathscinet-getitem?mr=\@MRSTRIP}{MR-\@MRSTRIP}}
\makeatother
\newcommand{\ARXIV}[1]{\href{http://arXiv.org/abs/#1}{arXiv:#1}}
\newcommand\D{\mathbb D}
\newcommand\Half{\mathbb H}
\newcommand\Prob{\mathbf P}
\newcommand\Disk{\D}
\newcommand\R{\mathbb R}
\newcommand\C{\mathbb C}
\newcommand\E{\mathbf E}
\newcommand\F{\mathcal F}

\newcommand\p\partial
\newcommand\inv{^{-1}}
\newcommand\ol\overline
\newcommand\sm\setminus
\newcommand\exc{\mathcal E}
\newcommand\rect{\mathcal R}
\newcommand\SLE{{SLE}}
\newcommand\SLEk{{SLE}$_\kappa$}
\renewcommand\Im{\operatorname{Im}}
\renewcommand\Re{\operatorname{Re}}
\DeclareMathOperator\diam{diam}
\DeclareMathOperator\dist{dist}

\title{Escape probability and transience for SLE\thanks{Research supported by NSF grant DMS-0907143.
    This is a retypeset copy of the article published in \href{http://doi.org/10.1214/EJP.v20-3714}{{\it Electronic Journal of Probability} {\bf 20} (2015), no. 10, 1--14}.}}

\author{Laurence S. Field\footnote{University of Chicago.}\and
Gregory F. Lawler\footnotemark[2]}

\date{February 12, 2015}

\begin{document}

\maketitle

\begin{abstract}
We give estimates for the probability 
that a chordal, radial or two-sided radial \SLEk\ curve 
retreats far from its terminal point after coming close to it, 
for $\kappa\le4$.
The estimates are uniform over all initial segments of the curve, 
and are sharp up to a universal constant.
\end{abstract}

%%                                                               %%
%% Please replace what follows by the body of your article       %%
%% (up to the bibliography):                                     %%
%%                                                               %%

\section{Introduction}

In this paper we will prove some estimates
 on the continuity of a Schramm--Loewner
evolution (\SLEk) curve at the terminal point.
 Before stating the estimates, we will describe
why they are useful.

We will consider three types of \SLEk.
These are probability measures on curves in a simply connected domain with specified start and end points.
Chordal \SLEk\ is a 
measure on curves from one boundary point to another.
Radial \SLEk\ is a measure
on curves from a boundary point to an interior point.
Two-sided radial \SLEk\ is essentially chordal \SLEk\ 
conditioned to pass through an interior point (the target point);
in this paper we will only consider the curve stopped when it hits that point.
The question we will be asking
is roughly ``given that the curve is very close to its target
point, what is the probability that it goes far
away before reaching the target point?''  There
are various parametrizations of the curves, but the questions
we discuss will be independent of the choice.

We will restrict our consideration to $0 < \kappa < 8$, for which the
curves have Hausdorff dimension $d := 1 + \frac \kappa 8 < 2$.
If $0 < \kappa \leq 4$, the measure is supported on simple 
curves while for $4 < \kappa < 8$, the curves have self-intersections.

  For chordal \SLEk, continuity at the endpoint is
  equivalent to transience of \SLEk\ in the 
  upper half-plane from $0$ to $\infty$,
   which was proved in \cite{RS}.
  We improve this result for $\kappa \leq 4$ by
  giving an upper bound for the probability of returning
  to the ball of radius~$\epsilon$ after the curve
  has reached the circle of radius~$1$.  It is a uniform
  estimate over all realizations of the curve up to
  the first visit to the unit circle.
  For the radial case, a similar result was stated
  in \cite{Lcont}. However, as pointed out by Dapeng Zhan, 
there is an error in one of the proofs.  We reprove this
estimate here discussing the error and how it is corrected.
We also give a uniform estimate for the two-sided radial case
which is analogous to the chordal case.   While
the endpoint continuity of two-sided radial \SLE\ was established in
\cite{Lcont}, and that of a generalization called radial SLE$_\kappa^\mu(\rho)$ in~\cite{MS}, no uniform estimates were given.

We will now make a precise statement of our
results.     Throughout this paper we let 
\[0 < \kappa <8
 \qquad\text{and let}\qquad
 a = \frac2\kappa.\] 
Let 
\[\begin{aligned}
\D_r &= \{ z\in\C: |z| < e^{-r} \},&
\quad C_r &= \p\D_r,\quad&
C &= C_0,\\
\Half &= \{ z\in\C: \Im z > 0 \}, &&\text{and}&
\D &= \D_0.
\end{aligned}\]
The convention in this paper is that letters $r$ and $s$ denote the negative logarithm of a disk's radius rather than the radius itself.
If $\gamma:[0,\infty) \rightarrow \C$ is a curve, we will
write $\gamma_t$ for the path $\gamma[0,t]$ and $\gamma(t)$
for the value of the curve at time~$t$.  We will write
$\gamma = \gamma_\infty$ for the entire path.

Suppose that $\gamma:[0,\infty) \to \ol\Half$ is a chordal \SLEk\
 curve from~$0$ to~$\infty$ in the half-plane~$\Half$.
The definition of chordal \SLEk\ states that $\gamma$ is a random
 curve characterized by the following property.  For
each $t$, let $H_t$ be the unbounded component of $\Half
\setminus \gamma_t$, and let~$g_t$ denote the unique conformal transformation
of $H_t$ onto~$\Half$ satisfying
$g_t(z) - z\rightarrow 0$ as~$z \rightarrow \infty$.  Then
$\gamma$ has been parametrized so that
\[   g_t(z) = z + \frac{at}{z} + O(|z|^{-2}), \qquad z \rightarrow
\infty , \]
and there exists a standard one-dimensional Brownian motion
$U$ such that 
\[    \p_tg_t(z) = \frac{a}{g_t(z) - U_t}, \qquad g_0(z) = z. \]
See \cite{Lbook} for more information.  If $\kappa \leq 4$, then
$\gamma$ is a simple curve, whereas if $\kappa \geq 8$, the curve
is plane-filling.

The uniform estimate for transience
is the following theorem.

\begin{theorem}  \label{thm1}
If $0 < \kappa \leq 4$, there exists $c < \infty$ such that if $\gamma$ is a
chordal \SLEk\ curve from $0$ to $\infty$ in $\Half$, 
$ T = \inf\{t:  \gamma(t)\in C\}$, and $r > 0$,
  then
\[   \Prob\{ \gamma[T,\infty) \cap C_r
 \neq \emptyset  \mid \gamma_T\} \leq c \,e^{-r(4a-1)}. \]
\end{theorem}

There is a similar result for radial \SLEk.  Let $\gamma$
be a radial \SLEk\ curve from~$1$ to~$0$ in the unit disk $\Disk$.
This is defined in terms of the conformal transformation
$h_t: \Disk \setminus \gamma_t \to \Disk$
with $h_t(0) =0$ and $h_t'(0)>0$.  Then the curve is parametrized
so that $h_t'(0) = e^{2at}$ and $L_t(z) := \frac 1{2i}\log h_t(e^{2iz})
$ satisfies 
\begin{equation}\label{radialparam}
   \p_t L_t(z) =  a  \cot(L_t(z) - U_t),
\qquad   L_0(z)=z,
\end{equation}
where $U$ is a standard Brownian motion.  If we write
$g_t(e^{2i\theta}) = e^{2 i V_t}$ and let $\Theta_t =
V_t - U_t$, then $\Theta_t$ satisfies
\[    d\Theta_t = a  \cot \Theta_t \, dt - dU_t.\]

\begin{theorem}   \label{thm2}  If $0 < \kappa \leq 4$, there exists
 $c <\infty$ such that if $\gamma$ is a radial 
\SLEk\ curve  from $1$ to $0$ in $\Disk$,
  $0 < s < r$, and $\tau_r = \inf\{t: \gamma(t) \in C_r\}$,
  then
\[   \Prob\{ \gamma[\tau_r,\infty) \cap C_s
 \neq \emptyset  \mid \gamma_{\tau_r}\} \leq c \,e^{-(r-s)(4a-1)/2}. \]
\end{theorem}

As mentioned before,
this  last result is stated in \cite{Lcont}; we prove it here 
by correcting the argument from that paper.

A two-sided radial \SLEk\ curve from $1$ to $e^{2i\theta}$ through $0$
in $\D$ can be thought of as chordal \SLEk\ from $1$ to $e^{2i\theta}$ in $\D$, conditioned to pass through~$0$.
For the purposes of this paper 
the curve will always be stopped when it reaches $0$.
The curve can be defined by weighting chordal \SLEk\ in the sense of the Girsanov theorem by the \SLEk\ Green's function in the slit domain at~$0$.
Explicitly,  the Green's function for \SLEk\ from $z$ to $w$ in a domain $D$
is
\[   G_D(\zeta;z,w) = \lim_{\epsilon \downarrow 0} \epsilon^{d-2}
  \, \Prob\{\dist(\zeta,\gamma) < \epsilon \} , \]
  where $\gamma$  is the entire curve.  The limit is known
  to exist \cite{LR,RS} and is given by
  \[      G_\Half(\zeta;0,\infty) = \hat c \, 
    [\Im \zeta]^{d-2} \, [\sin \arg \zeta]^{4a-1}, \]
    where $\hat c$ is a constant depending only on $\kappa$.
 For other domains, one uses the  conformal covariance rule
 \[     G_D(\zeta;z,w) = |f'(\zeta)|^{2-d}
    \, G_{f(D)}(f(\zeta);f(z),f(w)).\]
 In particular,
 \[     G_\Disk(0;1,e^{2i\theta}) = \hat c \, \sin^{4a-1}
  \, \theta.\]
  After weighting by the appropriate martingale, we see
  that 
\[
d\Theta_t=2a\cot\Theta_t\,dt+dW_t,\qquad
\Theta_0=\theta,\qquad
-dU_t=a\cot\Theta_t\,dt+dW_t,
\]
where $W$ is a standard Brownian motion in the new measure.
See~\cite{LW} for more details.

The two-sided radial estimate is almost the same as the radial estimate.
\begin{theorem}   \label{thm3}
If $0 < \kappa \leq 4$, there exists
 $c <\infty$ such that if $0<\theta<\pi$, $\gamma$ is a two-sided radial 
\SLEk\ curve  from $1$ to $e^{2i\theta}$ through $0$ in $\Disk$
stopped when it reaches $0$,
  $0 < s < r$, and $\tau_r = \inf\{t: \gamma(t) \in C_r\}$,
  then
\[   \Prob\{ \gamma[\tau_r,\infty) \cap C_s
 \neq \emptyset  \mid \gamma_{\tau_r}\} \leq c \,e^{-(r-s)(4a-1)/2}. \]
\end{theorem}

More generally, chordal, radial and two-sided radial \SLEk\ can be defined
in simply connected domains by conformal invariance.
If $D$ is a simply connected domain and $z \in \p D$, $w \in \overline D
\setminus \{z\}$, we will say \SLEk\ from $z$ to $w$ in $D$ with
the implication that it is chordal \SLEk\ if $w \in \p D$ and radial
\SLEk\ if $w \in D$. 
This convention is not just an arbitrary compression of notation, as the law of radial SLE targeted at an interior point converges to the law of chordal SLE as the point approaches the boundary; see~\cite[Section 8]{LL} for a precise formulation of this result.

\section{Definitions and notation}

Any time that we use the letter $c$, we will mean a finite positive constant, depending only on~$\kappa$, that may differ from all previous uses of $c$.
A \emph{crosscut} of a domain $D$ is a simple curve 
$\eta:(0,t_\eta)\to D$
with $\eta(0+),\eta(t_\gamma-)\in\p D$.
We will often  write just $\eta$
for the image $\eta(0,t_\eta)$.

We will take all domains to be in the Riemann sphere $\hat\C$,
and let $\hat\R=\R\cup\{\infty\}\subset\hat\C$.
Let $\R_+=(0,\infty)$ and $\R_-=(-\infty,0)$.
For a domain $D$, let $D^c=\hat\C\sm D$ denote its complement.
We always let~$B$ be a complex Brownian motion.
For a domain~$D$ and an arbitrary set~$S$, let
\[
\tau_D=\inf\{t:B_t\notin D\},\quad 
\sigma_S=\inf\{t:B_t\in S\}.
\]
We write $h_D$ for the harmonic measure, that is, if 
$V \subset \p D$, 
\[   h_D(z,V) = \Prob^z\{\tau_D \in V\}.\]

If $D$ is a domain and $V,W\subset\p D$ are analytic boundary arcs, then the {\em excursion measure} in~$D$ between~$V$ and~$W$ is defined as
\[
\exc_D(V,W)=\int_V\exc_D(v,W)\,|dv|,
\]
where, if $n_v$ is the inward pointing normal at $v$,
\[
\exc_D(v,W)=\lim_{\epsilon\downarrow 0}\epsilon\inv\,h_D(v+\epsilon n_v,W).
\]
It turns out that $\exc_D(V,W)$ is conformally invariant (see~\cite[Proposition 5.8]{Lbook}), 
and therefore it makes sense even when $V,W$ are boundary arcs that are not analytic. 
Also, the measure is symmetric,  $\exc_D(V,W)$ $= \exc_D(W,V)$.
As a conformal invariant, the excursion measure differs from the classical extremal length. One can see, however, that they determine one another in the case of a conformal rectangle: in the rectangle $D=[0,L]\times[0,\pi]$, the excursion measure between the sides of length $\pi$ is $e^{-L}$, whereas the extremal length of all curves joining these two sides is $L/\pi$.

We  extend this notation in two convenient ways.
First, if $D$ is not connected, we set $\exc_D(V,W)=\exc_{D_1}(V,W)$ where $D_1$ is the unique connected component of $D$ from which $V$ and $W$ are both accessible, or zero if there is no such component.
(Strictly speaking, this component is unique only once an orientation of the inward pointing normal is chosen on $V$ and on $W$, but this rarely needs to be specified.)
Second, if $V,W$ are not boundary arcs but merely the images of simple curves (that may pass through $D$) then we set $\exc_D(V,W)=\exc_{D\sm (V\cup W)}(V,W)$.

For fixed $V,W$, we can view $\exc_D(V,W)$ as a measure $\exc_{D\sm W}(V,\cdot)$ on $W$ of total
mass $\exc_D(V,W)$.  The
strong Markov property for Brownian motion gives the following rule.
Suppose $\eta$ is a curve in $D$ that separates $V$ and $W$.  Then
\[   \exc_{D}(V,W) = \int_{\eta} h_D(z,W) \, \exc_{D\sm \eta}(V,dz),\]
and hence,
\begin{align}
\exc_{D}(V,W) &\leq 
\exc_D(V,\eta) \, \sup_{z \in \eta} h_{D_1}(z,W), \label{new1}\\
\exc_{D}(V,W) &\geq 
\exc_{D\sm\eta}(V,\eta') \, \inf_{z \in \eta'} h_{D_1}(z,W)
\qquad\text{if}\quad \eta' \subset \eta. \label{new2}
\end{align}
In particular, we see that for any domain $D$ and any $r \geq 2$,
\begin{equation}  \label{jun21.2}
 \exc_D(C,C_r) \leq c \, \sup_{z \in C_1} h_{D \setminus
C_r} (z,C_r) . 
\end{equation}
The upper bound \eqref{new1}
implies the following two estimates for excursion measure.
Using \eqref{jun21.2}
 and the Poisson kernel in $\Half \setminus \D_r$, we can see that
\begin{equation}  \label{beurlight}
    \exc_\Half(C,C_r) \leq c \, e^{-r}. 
    \end{equation}
By the Beurling estimate (see, e.g., \cite[Theorem 3.76]{Lbook}),
there exists $c < \infty$ such that if $\C \setminus
D$ includes a curve connecting $C$ with $C_r$, then
\[  h_D(z, C_r) \leq c \, e^{-r/2}, \qquad z \in C_1,\]
and hence \eqref{jun21.2} gives
\begin{equation}  \label{beurl}
   \exc_D(C,C_r) \leq c \, e^{-r/2}.  
   \end{equation}

\section{Boundary intersection exponent for \SLE}

The following is the basic boundary intersection estimate for \SLE.  We will
state it in a unified form that combines radial and chordal cases.

\begin{proposition}     \label{boundaryest}
If $0 < \kappa < 8$, there exists $c < \infty$ such that
the following holds.  
Suppose $D$ is a simply connected domain, $z \in \p D$, and
$w \in \overline D \setminus \{z\}$.  
Suppose $\eta$ is a crosscut of $D$, and $\tilde \gamma:(0,1) \rightarrow D$ is a simple curve with 
$\tilde \gamma(0+) = z$, $\tilde \gamma(1-) = w$.  
If $\gamma$ is an \SLEk\ curve from $z$ to $w$ in $D$, then
\[    \Prob\{\gamma \cap \eta \neq \emptyset\}
   \leq c \, \exc_D(\eta,\tilde \gamma)^{4a-1}. \]
   \end{proposition}

\begin{proof}  By conformal invariance it suffices to consider the
chordal case with  $D = \Half$, $z=0$, $w=\infty$ and 
the radial case with $D = \Disk$, $z=1$, $w=0$.

For the chordal case, without loss of generality assume that the endpoints
of $\eta$ are on the positive real axis.  Then, standard estimates for excursion measure (e.g., Corollary~\ref{exc-cor}) show that
\[  \exc_\Half(\eta,\tilde \gamma) \geq \exc_\Half(\eta,\R_-)
  \geq c \, \Bigl(\frac{\diam(\eta)}{\dist(0,\eta)}\wedge1\Bigr). \]
A proof of 
\[  \Prob\{\gamma \cap \eta \neq \emptyset\}
   \leq c \,  \Bigl(\frac{\diam(\eta)}{\dist(0,\eta)}\wedge1\Bigr)^{4a-1} \]
   can be found in \cite{AK}.
   
For the radial case, if $\diam(\eta) \geq 1/10$, then $\exc_\Disk(\eta,\tilde \gamma) \geq c$ and the result is immediate.  
Hence we assume that $\diam(\eta)
\leq 1/10$.   A proof of
\[  \Prob\{\gamma \cap \eta \neq \emptyset\}
   \leq c \,  \Bigl(\frac{\diam(\eta)}{\dist(1,\eta)}\Bigr)^{4a-1} \]
   can be found in \cite{Lcont}.  Roughly speaking, the only difficult case
   is when $\eta$ is near 1, and in this case the path
   looks locally like chordal \SLE\ from $1$ to $-1$, for which
   the chordal estimate holds. By mapping to the  half-plane case,
   we can see that
\[  \exc_\Half(\eta,\tilde \gamma) \geq  
  c \, \frac{\diam(\eta)}{\dist(1,\eta)}. 
\qedhere 
\]
\end{proof}

It follows immediately that if $\eta_1,\eta_2,\ldots$ is a countable
collection of crosscuts and $A= \bigcup_j \eta_j $, then
\[    \Prob\{\gamma \cap A \neq \emptyset \}
   \leq c\,\sum_{j=1}^\infty \exc_D(\eta_j,\tilde \gamma)^{4a-1}. \]
If $0 < \kappa \leq 4$, then $4a-1 \geq 1$. This is exactly the condition that we need in order to make the estimate
\[
\sum_{j=1}^\infty \exc_D(\eta_j,\tilde \gamma)^{4a-1}
\le \left[\sum_{j=1}^\infty \exc_D(\eta_j,\tilde \gamma)
   \right] ^{4a-1},
   \]
 and hence we can
conclude the following.

\begin{proposition}   \label{mainprop}
If $0 < \kappa  \leq 4$, there exists $c < \infty$ such that
the following holds.  Suppose $D$ is a simply connected domain, $z \in \p D$,
$w \in \overline D \setminus \{z\}$.  Suppose $\eta_1,\eta_2,
\ldots $ are crosscuts of $D$.
Suppose $\tilde \gamma:(0,1) \rightarrow D$ is a simple curve with $\tilde \gamma(0+) = z$, $\tilde \gamma(1-) = w$.  
If~$\gamma$~is an \SLEk\ curve from $z$ to $w$ in $D$, then
\[    \Prob\left\{\gamma \cap  \bigcup_{j=1}^\infty
\eta_j  \neq \emptyset\right\}
   \leq c \, \left[\sum_{j=1}^\infty \exc_D(\eta_j,\tilde \gamma)
   \right] ^{4a-1}. \]
\end{proposition}
   
This was the strategy in \cite{Lcont} and there was no problem in the
argument at this point.  However, as pointed out by Dapeng Zhan,
when applying the argument, an  
 upper bound   was established for 
\begin{equation}  \label{11}
   \exc_D\left(\bigcup_{j=1}^\infty \eta_j ,\tilde \gamma\right) ,
   \end{equation}
rather than for 
\begin{equation}  \label{12}
 \sum_{j=1}^\infty \exc_D(\eta_j,\tilde \gamma).
 \end{equation}
In general, it may not be easy to bound \eqref{12} in terms of \eqref{11}.
Fortunately, we will need such a bound only when $\{\eta_j\}$ is the
set of crosscuts of $D$ contained in a particular circle
about the origin.  The next lemma shows that in this case 
 the quantity
\eqref{12} is at most twice the quantity \eqref{11}.

\begin{lemma}\label{sum-exc}
Let $D$ be a simply connected domain and $S$ the set of crosscuts of~$D$ that are subsets of the circle~$C_s$.
Then
\[
\sum_{\eta\in S}\exc_D(C_r,\eta)
\le 2\,\exc_D(C_r,C_s).
\]
\end{lemma}

\begin{proof}  See Section \ref{orange}. \end{proof}

Theorems~\ref{thm1} and~\ref{thm2} follow from the following propositions.  
We also  include
another case which arises in current research.  Recall
that $ \Disk_r = e^{-r} \, \Disk $
and $C_r = \p\Disk_r$. 

\begin{proposition} If $0 < \kappa \leq 4$, there exists
$c < \infty$ such that the following holds. 
Suppose that 
  $ D \subset \Half$ is a simply connected domain such
that $(\Half \setminus D)  \cap (\Half \setminus \Disk)  = \{z\}$.
Let $\gamma$ be an \SLEk\ curve in $D$ from $z$ to $\infty$.
Then,
\[  \Prob\{\gamma \cap C_r \neq \emptyset\} \leq c \, e^{-r(4a-1)}.\]
\end{proposition}

In particular, if $\kappa \leq 4$, $\gamma$ is an \SLEk\ curve
from $0$ to $\infty$ in $\Half$, 
and $\rho = \inf\{t: |\gamma(t) | = 1\}$, then
\[  \Prob\{\gamma(\rho,\infty) \cap C_r \neq \emptyset
\mid \gamma_\rho\} \leq c \, e^{-r(4a-1)}.\]
It is not known whether or not this estimate holds for
$4 < \kappa < 8$.

\begin{proof}  We may assume $r \geq 1$. We write 
\[   D \cap C_r = \bigcup_{j=1}^ \infty \eta_j , \]
where $\eta_j$ are crosscuts of $D$.  Let $\tilde \gamma$ be
any simple curve from~$z$ to~$\infty$ in~$\Half$
 that does not intersect
$C \setminus \{z\}$.  Then for each $j$,
\[     \exc_D(\eta_j, \tilde \gamma) \leq \exc_D(\eta_j,
   C), \]
   and hence, using \eqref{beurlight},
\[      \sum_{j=1}^\infty \exc_D(\eta_j,
  \tilde \gamma ) \leq 
    \sum_{j=1}^\infty \exc_D(\eta_j,
   C) \leq 2 \,\exc_D(C_r,C)
    \leq 2 \,\exc_\Half(C_r,C) \leq c \, e^{-r}.
\]
The second inequality above uses Lemma~\ref{sum-exc}.
\end{proof}

\begin{proposition}\label{prob-int-Cs}
If $0 < \kappa \leq 4$, there exists $c
< \infty$ such that  if $0 < s \leq r$, 
 $D\subset\D$ is a simply connected domain,
and one of the following hold:
\begin{itemize}
\item
 $\p D\cap\ol\D_r=\{z,w\}$,  $z\ne w$, and $\gamma$ is an \SLEk\ curve in $D$ from~$z$ to~$w$,
\item   $\p D\cap\ol\D_r=\{z\}$, and $\gamma$ is an \SLEk\ curve in $D$ from $z$ to~$0$,\end{itemize}
then 
\[  \Prob\{\gamma \cap C_s \ne\emptyset\} 
   \leq c\,e^{-(4a-1)(r-s)/2}.\]
\end{proposition}

\begin{proof} We may assume that $r \geq s+1$.
Let $\tilde \gamma$ be a straight line from $z$ to $w$ (in the first
case) or $0$ (in the second case). 
Let $S$ be the set of crosscuts of~$D$ that are subsets of the circle~$C_s$.
Then for each $\eta \in S$, $\exc_D(\eta, \tilde \gamma)
\leq \exc_D(\eta,C_r)$.  Therefore, using \eqref{beurl},
\[ \sum_{\eta \in S} \exc_D(\eta,\tilde \gamma)
  \leq \sum_{\eta\in S} \exc_D(\eta,C_r) \leq 
     2 \, \exc_D(C_s, C_r) \leq 
       c  \, e^{-(r-s)/2} . \]
      The middle inequality uses Lemma~\ref{sum-exc} and the last inequality uses the Beurling estimate.
The result then follows from Proposition \ref{mainprop}.
\end{proof}

\section{Proof of Lemma \ref{sum-exc}}  \label{orange}

Here we prove Lemma \ref{sum-exc}.  The idea is simple, and is illustrated in Figure~\ref{fig:bm-crosscuts}. 
Suppose $D$ is a simply connected domain and $\eta$
is a crosscut of $D$ contained in a circle $C_s$.  
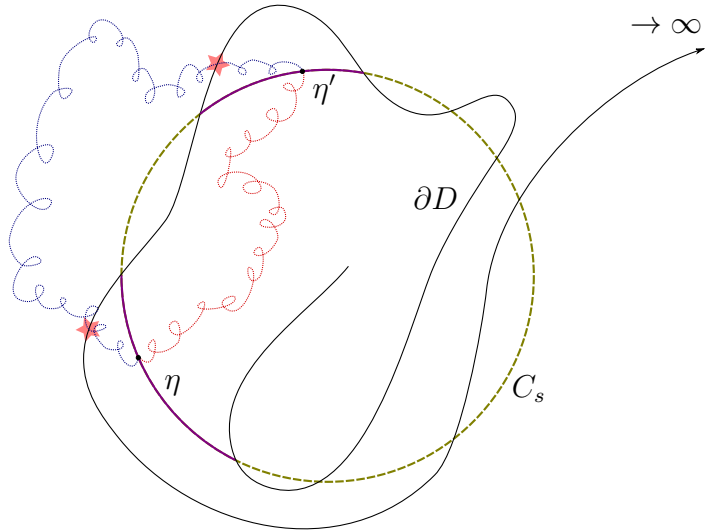
\begin{figure}[h]
\centerline{%% Creator: Inkscape inkscape 0.48.4, www.inkscape.org
%% PDF/EPS/PS + LaTeX output extension by Johan Engelen, 2010
%% Accompanies image file 'bm-crosscuts.pdf' (pdf, eps, ps)
%%
%% To include the image in your LaTeX document, write
%%   \input{<filename>.pdf_tex}
%%  instead of
%%   \includegraphics{<filename>.pdf}
%% To scale the image, write
%%   \def\svgwidth{<desired width>}
%%   \input{<filename>.pdf_tex}
%%  instead of
%%   \includegraphics[width=<desired width>]{<filename>.pdf}
%%
%% Images with a different path to the parent latex file can
%% be accessed with the `import' package (which may need to be
%% installed) using
%%   \usepackage{import}
%% in the preamble, and then including the image with
%%   \import{<path to file>}{<filename>.pdf_tex}
%% Alternatively, one can specify
%%   \graphicspath{{<path to file>/}}
%% 
%% For more information, please see info/svg-inkscape on CTAN:
%%   http://tug.ctan.org/tex-archive/info/svg-inkscape
%%
\begingroup%
  \makeatletter%
  \providecommand\color[2][]{%
    \errmessage{(Inkscape) Color is used for the text in Inkscape, but the package 'color.sty' is not loaded}%
    \renewcommand\color[2][]{}%
  }%
  \providecommand\transparent[1]{%
    \errmessage{(Inkscape) Transparency is used (non-zero) for the text in Inkscape, but the package 'transparent.sty' is not loaded}%
    \renewcommand\transparent[1]{}%
  }%
  \providecommand\rotatebox[2]{#2}%
  \ifx\svgwidth\undefined%
    \setlength{\unitlength}{262.24638039bp}%
    \ifx\svgscale\undefined%
      \relax%
    \else%
      \setlength{\unitlength}{\unitlength * \real{\svgscale}}%
    \fi%
  \else%
    \setlength{\unitlength}{\svgwidth}%
  \fi%
  \global\let\svgwidth\undefined%
  \global\let\svgscale\undefined%
  \makeatother%
  \begin{picture}(1,0.75452129)%
    \put(0,0){\includegraphics[width=\unitlength]{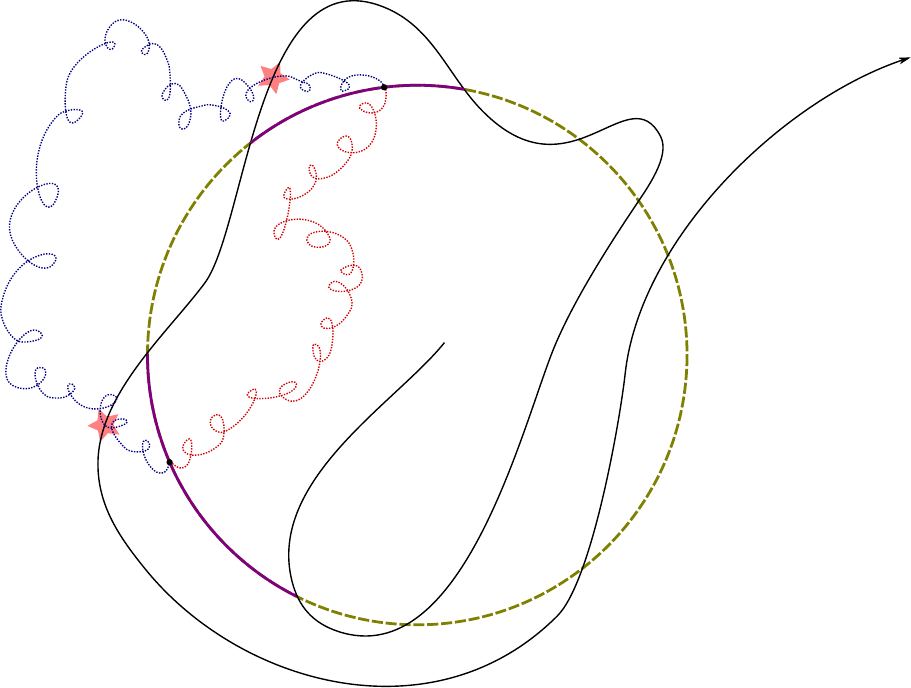}}%
    \put(0.72313999,0.19040979){\color[rgb]{0,0,0}\makebox(0,0)[lb]{\smash{$C_s$}}}%
    \put(0.64582217,0.46055901){\color[rgb]{0,0,0}\makebox(0,0)[rb]{\smash{$\partial D$}}}%
    \put(0.43659519,0.62047244){\color[rgb]{0,0,0}\makebox(0,0)[lb]{\smash{$\eta'$}}}%
    \put(0.2240782,0.20140479){\color[rgb]{0,0,0}\makebox(0,0)[lb]{\smash{$\eta$}}}%
    \put(0.996663,0.71398299){\color[rgb]{0,0,0}\makebox(0,0)[rb]{\smash{$\to\infty$}}}%
  \end{picture}%
\endgroup%
}
\caption{The idea behind Lemma~\ref{sum-exc}.}
\label{fig:bm-crosscuts}
\end{figure}
Suppose
we start a Brownian motion path on $\eta$ and stop it when it reaches
another crosscut $\eta'$ on $C_s$.  Consider the union of the
path and its reflection about $C_s$.  If the union is
in $D$, then we have disconnected the boundary, contradicting
the simple connectedness of~$D$.  Hence, for every
such path, either the path or its reflection hits $\p D$
before reaching the new crosscut.
In this section we make
this argument precise.

\begin{proposition}\label{at-most-half}
Let $D$ be a domain such that there is only one connected component of $D^c$ that intersects $\hat\R$.
Let $S$ denote the set of crosscuts of~$D$ that are subsets of~$\hat\R$, so that $D\cap\hat\R=\bigcup_{\eta\in S}\eta$.
Let $x\in D\cap\hat\R$ be contained in crosscut $\eta_x$, let $S_x = S \setminus \{\eta_x\}$ 
 and let 
\[  E_x = \bigcup _{\eta \in S_x}
\{ \sigma_\eta<\tau_D\} \]
be the event that the Brownian motion reaches a crosscut other than $\eta_x$
before leaving the domain~$D$.
Then $\Prob^x(E_x)\le 1/2$.
\end{proposition}

We note the following.
\begin{itemize}
\item The conditions of the proposition hold if 
 $D$ is simply connected, or  if $D$ is the intersection of a simply connected domain with a domain containing $\hat\R$.

\item The constant $1/2$ is optimal as can be seen by considering 
  $D=\C\setminus e^{i\theta}[0,\infty)$ for $0<\theta<\pi$.
If $x>0,$ then $\Prob^x\{\sigma_{(-\infty,0)}<\tau_D\}=\theta/(\theta+\pi)\to1/2$ as $\theta\to\pi$.
\end{itemize}

\begin{proof}
Let $\ol D=\{\ol z:z\in D\}$ be the reflection of $D$ about
the real line. 
Let $\ol E_x$ be the event 
that $\sigma_\eta<\tau_{\ol D}$ for some $\eta\in S_x$.
In other words, this is the event $E_x$ applied to the complex conjugate $\ol{B_t}$ rather than $B_t$.
Since Brownian motion starting from $x$ is invariant under complex conjugation, $\Prob^x(E_x)=\Prob^x(\ol E_x)$.

On the event $E_x\cap\ol E_x$, the path 
$B[0,\sigma_\eta]\cup \ol B[0,\sigma_\eta]$ lies entirely within $D$, but separates one endpoint of $\eta$ from the other.
But this is impossible, as the two endpoints of~$\eta$ lie in the same connected complement of $D^c$.
Hence $\Prob^x(E_x\cap\ol E_x)=0$, and so $\Prob^x(E_x)\le1/2$.
\end{proof}

\begin{corollary}\label{E-nr-crosscuts}
Let $D$ be a domain such that there is only one component of $D^c$ that intersects $C$.
Let $S$ be the set of crosscuts of~$D$ that are subsets of the unit circle~$C$.
Let~$\tau=\tau_D$.
Let 
\[    V = \sum_{\eta\in S} 1\{\sigma_\eta < \tau\} \]
be the number of distinct crosscuts visited by the Brownian motion
by time~$\tau$. 
Then for any $z\in D$,
\[
\E^z [V]
\le 2\, \Prob^z\{\sigma_C<\tau\}.
\]
\end{corollary}

\begin{proof}
First observe that
\begin{equation*}
\E^z[V] 
= \E^z\{V \mid\sigma_C<\tau\}
\,\Prob^z\{\sigma_C<\tau\} \leq \Prob^z \{\sigma_C < \tau\} \, \sup_{w \in C}
   \E^w[V].
   \end{equation*}
Consider a M\"obius transformation of the Riemann sphere 
that sends~$C$ to~$\R$.
Since Brownian motion is conformally invariant,
it follows from Proposition~\ref{at-most-half} that, 
for all~$w\in C$,
\[
\Prob^w\{\tau<\sigma_\eta
\text{ for all }\eta\in S\text{ with }w\notin\eta\}
\ge 1/2,
\]
Therefore, the number of crosscuts hit before leaving $D$
by a Brownian motion started from any $w\in C$
is stochastically dominated by a geometric random variable 
of parameter~$1/2$,  and hence $\E^w[V] \leq 2 . $
\end{proof}

\begin{proof}[Proof of Lemma \ref{sum-exc}]
We may assume without loss of generality that $r>s$.
Let $D_1=D\sm\ol\D_r$.
Using the definition of excursion measure, it suffices to show that, for any $z\in C_r$ and $\epsilon>0$ small, if $w=z(1+\epsilon)$,
\begin{equation}\label{sum-of-PKs}
\sum_{\eta\in S} h_{D_1\sm\eta}(w,\eta)
\le 2\, h_{D_1\sm C_s}(w,C_s)
\end{equation}
Let $D_2$ be the connected component of $D_1$ that contains $w$, so that we may replace $D_1$ by $D_2$ in~\eqref{sum-of-PKs}.
Now either $\ol\D_r\subset D$, in which case $D_2^c$ consists of $\ol\D_r$ (which does not intersect $C_s$) and one other component; or $\ol\D_r\not\subset D$, in which case $D_1^c$ has only one connected component and hence $D_2$ is simply connected.
In either case we see that Corollary~\ref{E-nr-crosscuts} is applicable and says that
\[
\sum_{\eta\in S}\Prob^w\{\sigma_\eta<\tau_{D_2}\}
\le 2\, \Prob^w\{\sigma_{C_s}<\tau_{D_2}\}.
\]
This yields~\eqref{sum-of-PKs}.
\end{proof}

\section{Two-sided radial \SLE}

With  Lemma \ref{sum-exc} the proofs of Theorems \ref{thm1}
and \ref{thm2} are complete.  
In this section we will prove Theorem \ref{thm3}.

% The next lemma follows from standard estimates for Brownian motion,
% but for completeness we include a proof.

\begin{lemma}\label{exc-half}
Let $\gamma:[0,1]
\rightarrow \overline \Half$ be a simple curve with 
$\gamma(0) = 0$ and $\max_t|\gamma(t)|=r\le 1/4$.
Let $\eta:[0,1] \rightarrow \overline
\Half$ be a simple curve  with $\eta(0) \in \R$ and
$\dist(0,\eta)=1$.
Then
\[
\exc_\Half(\gamma,\eta)\asymp r\,\bigl(\diam(\eta)\wedge1\bigr).
\]
\end{lemma}

\begin{proof}
   Let $D = \{z \in \Half: |z| < 1/2\}$.
If $|z| = 1/2$, one can deduce from basic properties of Brownian motion that  
\begin{align*}
\exc_{D\setminus \gamma}(z,\gamma) &\asymp  \Im(z) \, \diam (\gamma)
\asymp r \, \Im(z), \\
h_{\Half \setminus (\gamma\cup\eta)}(z,\eta) 
&\asymp \Im(z) \, \bigl(\diam(\eta)\wedge1\bigr).
\end{align*}
Hence the result follows from \eqref{new1} and \eqref{new2}.
\end{proof}

From this lemma we deduce an estimate for the excursion measure in the half-plane between two curves that touch the boundary.

\begin{corollary}\label{exc-cor}
If $\gamma,\eta:[0,1]\to\ol\Half$ are disjoint simple curves with $\gamma(0),\eta(0)\in\R$, then
\[
\exc_\Half(\gamma,\eta)\wedge1 \asymp
\Bigl(\frac{\diam(\gamma)}{\dist(\gamma,\eta)}\wedge1\Bigr)
\Bigl(\frac{\diam(\eta)}{\dist(\gamma,\eta)}\wedge1\Bigr).
\]
\end{corollary}

\begin{proof}
We may assume without loss of generality that $\diam(\gamma)\le\diam(\eta)$, $\gamma(0)=0$ and $\dist(0,\eta)=1$.
Let $T=\inf\{t:\gamma(t)=1/4\}$.

If $T=\infty$ then $\dist(\gamma,\eta)\in(3/4,1]$ and hence Lemma~\ref{exc-half} gives the claim.

If $T<\infty$, Lemma~\ref{exc-half} shows that
$\exc_\Half(\gamma[0,T],\eta)\asymp 1$, which implies the claim since $\dist(\gamma,\eta)\le1$.
\end{proof}

\begin{lemma}\label{cdE}
There exist $\delta >0$ and $c<\infty$ such that
if $\hat\eta$ is a crosscut of $\Half$
and $\hat I$ is a simple curve in $\overline\Half$ from $0$ to $e^{i\theta}$
with $\exc_\Half(\hat I,\hat\eta) \leq \delta$, then
\[
\frac{\diam(\hat\eta)}{\dist(\hat\eta,0)}\wedge1
<c\,\dist(e^{i\theta},\hat\eta)\,\exc_\Half(\hat I,\hat\eta).
\]
\end{lemma}

\begin{proof}
Write $d=\dist(e^{i\theta},\hat\eta)$,
$d_0=\dist(\hat\eta,0)$ and $\exc=\exc_\Half(\hat I,\hat\eta)$.

By Corollary~\ref{exc-cor}, 
there is a universal constant $c_0\in(0,1)$ such that
\begin{equation}\label{capped}
\exc\wedge1 > c_0
\Bigl(\frac1d\wedge1\Bigr)
\Bigl(\frac{\diam(\hat\eta)}{d\wedge d_0}\wedge1\Bigr).
\end{equation}

If $d\ge1$, the estimate~\eqref{capped} shows that
\[
\exc > c_0\,\frac1d
\Bigl(\frac{\diam(\hat\eta)}{d_0}\wedge1\Bigr),
\]
and the conclusion follows.

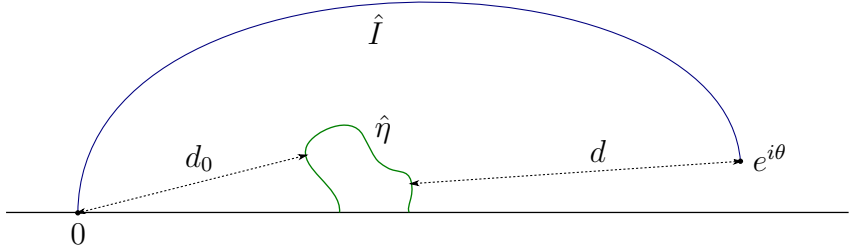
\begin{figure}[ht]
\centerline{%% Creator: Inkscape inkscape 0.48.4, www.inkscape.org
%% PDF/EPS/PS + LaTeX output extension by Johan Engelen, 2010
%% Accompanies image file 'cdE-lemma.pdf' (pdf, eps, ps)
%%
%% To include the image in your LaTeX document, write
%%   \input{<filename>.pdf_tex}
%%  instead of
%%   \includegraphics{<filename>.pdf}
%% To scale the image, write
%%   \def\svgwidth{<desired width>}
%%   \input{<filename>.pdf_tex}
%%  instead of
%%   \includegraphics[width=<desired width>]{<filename>.pdf}
%%
%% Images with a different path to the parent latex file can
%% be accessed with the `import' package (which may need to be
%% installed) using
%%   \usepackage{import}
%% in the preamble, and then including the image with
%%   \import{<path to file>}{<filename>.pdf_tex}
%% Alternatively, one can specify
%%   \graphicspath{{<path to file>/}}
%% 
%% For more information, please see info/svg-inkscape on CTAN:
%%   http://tug.ctan.org/tex-archive/info/svg-inkscape
%%
\begingroup%
  \makeatletter%
  \providecommand\color[2][]{%
    \errmessage{(Inkscape) Color is used for the text in Inkscape, but the package 'color.sty' is not loaded}%
    \renewcommand\color[2][]{}%
  }%
  \providecommand\transparent[1]{%
    \errmessage{(Inkscape) Transparency is used (non-zero) for the text in Inkscape, but the package 'transparent.sty' is not loaded}%
    \renewcommand\transparent[1]{}%
  }%
  \providecommand\rotatebox[2]{#2}%
  \ifx\svgwidth\undefined%
    \setlength{\unitlength}{338.87329102bp}%
    \ifx\svgscale\undefined%
      \relax%
    \else%
      \setlength{\unitlength}{\unitlength * \real{\svgscale}}%
    \fi%
  \else%
    \setlength{\unitlength}{\svgwidth}%
  \fi%
  \global\let\svgwidth\undefined%
  \global\let\svgscale\undefined%
  \makeatother%
  \begin{picture}(1,0.2733457)%
    \put(0,0){\includegraphics[width=\unitlength]{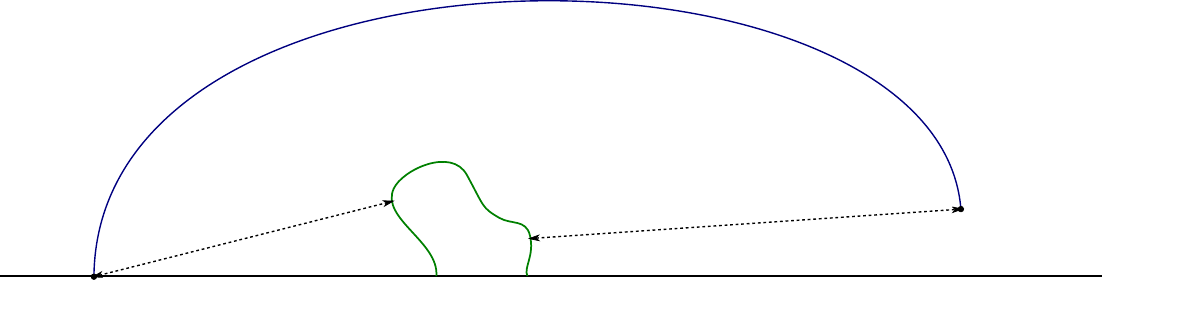}}%
    \put(0.08043947,0.00346968){\color[rgb]{0,0,0}\makebox(0,0)[b]{\smash{$0$}}}%
    \put(0.83049297,0.08496137){\color[rgb]{0,0,0}\makebox(0,0)[lb]{\smash{$e^{i\theta}$}}}%
    \put(0.40194031,0.22664873){\color[rgb]{0,0,0}\makebox(0,0)[lb]{\smash{$\hat I$}}}%
    \put(0.40998844,0.122405){\color[rgb]{0,0,0}\makebox(0,0)[lb]{\smash{$\hat\eta$}}}%
    \put(0.66699389,0.09530347){\color[rgb]{0,0,0}\makebox(0,0)[rb]{\smash{$d$}}}%
    \put(0.23094616,0.08958203){\color[rgb]{0,0,0}\makebox(0,0)[rb]{\smash{$d_0$}}}%
  \end{picture}%
\endgroup%
}
\caption{The case $d<1$ in Lemma~\ref{cdE}.}
\label{fig:cdE-lemma}
\end{figure}

Suppose that $d<1$, a case which is illustrated in Figure~\ref{fig:cdE-lemma}.
We will assume that $\exc\le\delta:=c_0/4$.
Then~\eqref{capped} implies that
\begin{equation}\label{uncapped}
\exc > c_0\,\frac{\diam(\hat\eta)}{d}
\end{equation}
and hence $\diam(\hat\eta)<d\,\exc/c_0<1/4$.
If $d_0>1/4$, the conclusion follows directly from~\eqref{uncapped}.
Otherwise, we see that $1/2<d<1$, in which case the conclusion follows from~\eqref{capped}.
\end{proof}

In~\cite{Lcont}, the boundary estimate for radial \SLEk\ 
is first proved for a finite time, and then extended to infinite time.
We follow the same strategy here 
to prove the boundary estimate for two-sided radial \SLEk.

Given a domain $D$ and $z,w\in\p D$, 
let $g$ be a conformal map sending $D,z,w$ to $\Half,0,\infty$.
Recall that the \SLEk\ Green's function at $\zeta\in D$ is
\[
G_D(\zeta;z,w)=\hat c\,
\Bigl(\frac{\Im g(\zeta)}{|g'(\zeta)|}\Bigr)^{d-2}\,S_D(\zeta;z,w)^{4a-1},
\]
where $S_D(\zeta;z,w)=\sin\arg g(\zeta)$.  
As before, we write $\gamma_t$ for the image  $\gamma[0,t]$.

\begin{proposition}\label{2sr-finite-time}
If $0<\kappa<8$, there exists $c<\infty$ such that the following holds for all~$\theta$.
Let $\gamma$ be a two-sided radial \SLEk\ curve in $\D$ from $1$ to $e^{-2i\theta}$ through~$0$ run until the stopping time $\rho=\inf\{t:|\gamma(t)|=\frac1{16}\}$.
Let $\eta$ be a crosscut of $\D$.
Then
\[
\Prob\{\gamma[0,\rho]\cap\eta\ne\emptyset\}
\le c\,\Bigl(\frac{\diam(\eta)}{\dist(\eta,1)}\Bigr)^{4a-1}.
\]
\end{proposition}

\begin{proof}
It suffices to prove the result for 
$\diam(\eta) / \dist(\eta,1)  $ sufficiently small.
For notational convenience we will write $\exc=\exc_\D([0,1],\eta)$ and $\alpha=1/(4a-1)$. 
Recall that $\exc\wedge1\asymp\bigl(\diam(\eta)/\dist(\eta,1)\bigr)\wedge1$.
Hence it suffices to prove the result for $\exc  \leq \delta$
  where~$\delta$~is as in  Lemma~\ref{cdE}
and  $\diam(\eta) \leq (1/100) \, \dist(\eta,1)
 \leq 1/50$.   We assume throughout that $\eta$
 satisfies this and hence that
  $\exc \asymp \diam(\eta) /  \dist(\eta,1) $.

Write $\tau=\inf\{t:\gamma(t)\in\eta\}$.
Let $\Prob_0$ be a law under which $\gamma$ is a chordal \SLEk\ curve in~$\D$ from $1$ to $e^{-2i\theta}$, once again stopped on reaching $\frac1{16}C$.
It follows from the definition of two-sided radial \SLEk\  that for any stopping time $\sigma\le\rho$,
\[
\frac{d\Prob}{d\Prob_0}(\gamma_\sigma)=1\{\sigma<\infty\}\,
\frac{G_\sigma}{G_0}\asymp 1\{\sigma<\infty\}\,
\Bigl(\frac{S_{\sigma}}{S_0}\Bigr)^{4a-1}.
\]
since the conformal radius never changes by more than a factor of $16$. Here we have denoted $G_\sigma=G_{\D\sm\gamma_\sigma}(0;\gamma(\sigma),e^{-2i\theta})$ and $S_\sigma$ similarly.

There is a unique conformal transformation $g$   sending $\D$ to $\Half$, $1$ to $0$, 
$e^{-2i\theta}$ to $\infty$ and $0$ to $e^{i\theta}$.  
Indeed, this is just the fractional linear transformation
\[ g(z)=e^{-i\theta}\frac{z-1}{z-e^{-2i\theta}}. \]
We write $\hat I=g([0,1])$, $\hat\eta=g(\eta)$, $\hat\gamma=g\circ\gamma$,
and let $K_\theta = g(\{|z| \leq 1/16\})$. 
By a reflection we may assume that $0<\theta\le\pi/2$.  Note
that $|g'(0)| = 2 \, \sin \theta$; this can be seen be direct
calculation or by noting that $|g'(0)|$ gives the 
conformal radius of $\Half$ with respect to $e^{i\theta}$.
Using the Koebe $1/4$ theorem on $g^{-1}$ we can see
that $K_\theta$ is contained in
the disk of radius $[\sin \theta]/2$ about $e^{i\theta}$, and
hence in 
  the disk of
radius $2 \theta$ about $1$. By the definition of~$\rho$,
 we have that $\rho = \inf\{t: \hat \gamma(t) \in K_\theta \}$. 

The argument proceeds in two cases depending on the position of $\hat\eta$ in $\Half$.
Let  \[d=\dist(e^{i\theta},\hat\eta).\]

\medskip
{\bf Case 1 (trapped case): }
$d>100\,\theta$ and the endpoints of $\hat\eta$ lie on $(0,1)$.
%  
% In this case, note that $\theta < 1/100$ and the Koebe $1/4$ theorem
% can be used to see that $K_\theta$ is mapped into the
% disk of radius~$2\theta$ about~$1$.  

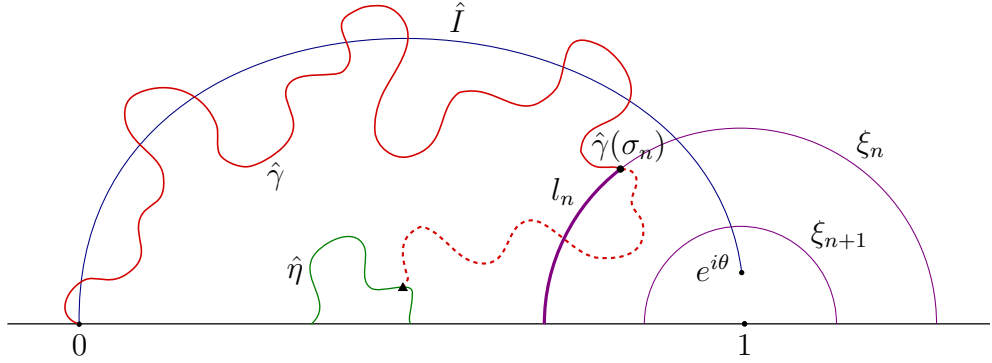
\begin{figure}[h]
\centerline{%% Creator: Inkscape inkscape 0.48.4, www.inkscape.org
%% PDF/EPS/PS + LaTeX output extension by Johan Engelen, 2010
%% Accompanies image file 'trappedcase.pdf' (pdf, eps, ps)
%%
%% To include the image in your LaTeX document, write
%%   \input{<filename>.pdf_tex}
%%  instead of
%%   \includegraphics{<filename>.pdf}
%% To scale the image, write
%%   \def\svgwidth{<desired width>}
%%   \input{<filename>.pdf_tex}
%%  instead of
%%   \includegraphics[width=<desired width>]{<filename>.pdf}
%%
%% Images with a different path to the parent latex file can
%% be accessed with the `import' package (which may need to be
%% installed) using
%%   \usepackage{import}
%% in the preamble, and then including the image with
%%   \import{<path to file>}{<filename>.pdf_tex}
%% Alternatively, one can specify
%%   \graphicspath{{<path to file>/}}
%% 
%% For more information, please see info/svg-inkscape on CTAN:
%%   http://tug.ctan.org/tex-archive/info/svg-inkscape
%%
\begingroup%
  \makeatletter%
  \providecommand\color[2][]{%
    \errmessage{(Inkscape) Color is used for the text in Inkscape, but the package 'color.sty' is not loaded}%
    \renewcommand\color[2][]{}%
  }%
  \providecommand\transparent[1]{%
    \errmessage{(Inkscape) Transparency is used (non-zero) for the text in Inkscape, but the package 'transparent.sty' is not loaded}%
    \renewcommand\transparent[1]{}%
  }%
  \providecommand\rotatebox[2]{#2}%
  \ifx\svgwidth\undefined%
    \setlength{\unitlength}{370.3bp}%
    \ifx\svgscale\undefined%
      \relax%
    \else%
      \setlength{\unitlength}{\unitlength * \real{\svgscale}}%
    \fi%
  \else%
    \setlength{\unitlength}{\svgwidth}%
  \fi%
  \global\let\svgwidth\undefined%
  \global\let\svgscale\undefined%
  \makeatother%
  \begin{picture}(1,0.36024476)%
    \put(0,0){\includegraphics[width=\unitlength]{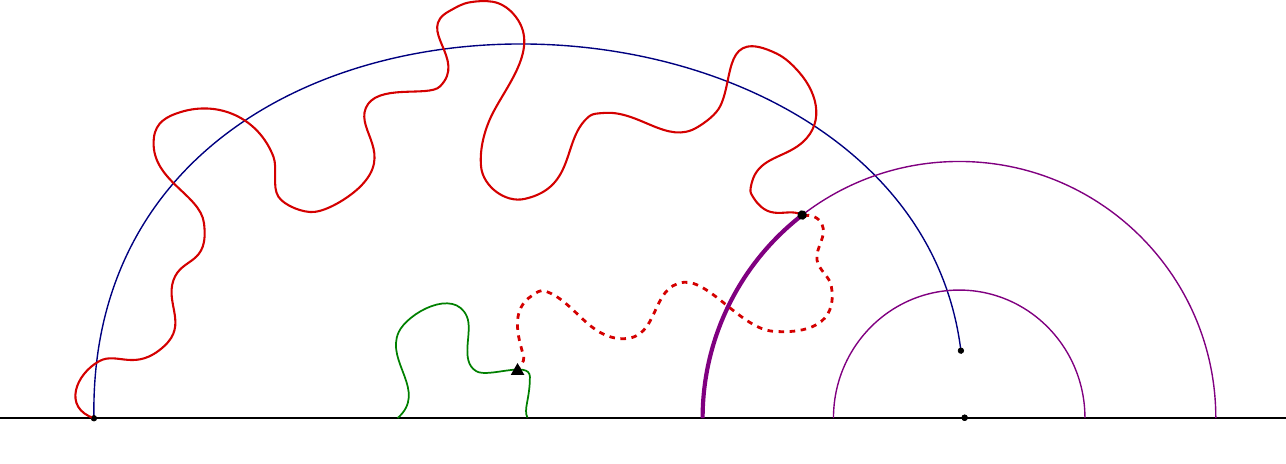}}%
    \put(0.07361272,0.00317521){\color[rgb]{0,0,0}\makebox(0,0)[b]{\smash{$0$}}}%
    \put(0.73587617,0.07929401){\color[rgb]{0,0,0}\makebox(0,0)[rb]{\smash{$e^{i\theta}$}}}%
    \put(0.44961532,0.33240864){\color[rgb]{0,0,0}\makebox(0,0)[lb]{\smash{$\hat I$}}}%
    \put(0.30124362,0.08234917){\color[rgb]{0,0,0}\makebox(0,0)[rb]{\smash{$\hat\eta$}}}%
    \put(0.57798114,0.16144077){\color[rgb]{0,0,0}\makebox(0,0)[rb]{\smash{$l_n$}}}%
    \put(0.86412397,0.21005133){\color[rgb]{0,0,0}\makebox(0,0)[lb]{\smash{$\xi_n$}}}%
    \put(0.81772262,0.11858173){\color[rgb]{0,0,0}\makebox(0,0)[lb]{\smash{$\xi_{n+1}$}}}%
    \put(0.59443212,0.20850031){\color[rgb]{0,0,0}\makebox(0,0)[lb]{\smash{$\hat\gamma(\sigma_n)$}}}%
    \put(0.26412521,0.17811763){\color[rgb]{0,0,0}\makebox(0,0)[lb]{\smash{$\hat\gamma$}}}%
    \put(0.7502657,0.00317521){\color[rgb]{0,0,0}\makebox(0,0)[b]{\smash{$1$}}}%
  \end{picture}%
\endgroup%
}
\caption{The trapped case in Proposition~\ref{2sr-finite-time}.}
\label{fig:trapped}
\end{figure}

Let $\sigma_n=\inf\{t:|\hat\gamma(t)-1|=2^{-n}\,d\}$.
 Define integer $k$ by $ 2 \theta \leq 2^{-k}\, d < 4 \theta$
 and note that $\sigma_k < \rho$ if $\sigma_k < \infty$. 
For $n > 0$, let $V_n$ denote the event $\{\sigma_n < \infty, \sigma_n < \tau\}$. 
On the event $V_n$, let  $\xi_n$ be the circle of radius $2^{-n}\,d$
 about~$1$, 
 $l_n$~the circular arc of $\xi_n$
  from $\hat\gamma(\sigma_n)$ 
to $1 - 2^{-n} \, d$, and  $H_n = \Half \setminus (\hat\gamma_{\sigma_n}
 \cup l_n)$. 
The boundary estimate (Proposition \ref{boundaryest}) implies
that \[\Prob_0 (V_n)^\alpha  \leq
\Prob_0\{\sigma_n<\infty\}^\alpha\asymp 2^{-n}\,d.\]
The curve $l_n$ disconnects $\hat \eta$ from infinity,
and hence 
on the event $V_n$, 
\[
\Prob_0\{\tau<\infty\mid\F_{\sigma_n}\}^\alpha
< c\,\exc_{H_n}( l_n, \hat \eta ).
\]
Using Corollary~\ref{exc-cor}, we can see that if $n\ge1$,
\[ \exc_{H_n}( l_n, \hat \eta )
 \leq \exc_{\Half}(\xi_n, \hat \eta ) \leq  
 c \,  \frac{\diam(\hat \eta)}{\dist(\hat \eta,1)} \, 2^{-n}
  \leq c \, 2^{-n} \, \exc.\]
Therefore,
\[
\Prob_0\{\sigma_n<\tau<\rho\}^\alpha
<c\,2^{-n}\,d\cdot 2^{-n}\,\exc.
\]
Note that $S_{\sigma_{n+1}}<c\,\theta/(2^{-n}\,d) \asymp S_0/(2^{-n}\,d)$.
It follows that
\[
\Prob\{\sigma_n<\tau<\sigma_{n+1}\}^\alpha < c\,2^{-n}\,\exc, \qquad n < k .
\]
Since $S_{\rho}  \leq 1\asymp S_0/(2^{-k}\,d)$, we also get
\[ 
\Prob\{\sigma_k<\tau<\rho \}^\alpha < c\,2^{-k}\,\exc .
\]
Similarly, $\Prob_0\{\tau<\infty\}^\alpha<c\,d\,\exc$ and $S_{\sigma_1}<c\,\theta/d\asymp S_0/d$ if $\tau<\sigma_1$, and therefore $\Prob\{\tau<\sigma_1\}^\alpha<c\,\exc$.
Adding the estimates so obtained, we see that
\[
\Prob\{\tau<\rho\}^\alpha<c\,\exc.
\]

\medskip

{\bf Case 2 (untrapped case): }
any pair $(\theta,\hat\eta)$ not covered in Case 1.

In this case, it suffices to show that 
\begin{align}
\Prob_0\{\tau<\infty\}^\alpha&<c\,d\,\exc,\label{one}\\
S_\tau&<c\,\theta/d,\label{two}
\end{align}
for we then have that $\Prob\{\tau<\rho\}^\alpha<c\,\exc$.
Estimate \eqref{one} follows, in light of Lemma~\ref{cdE}, 
from  the basic boundary estimate for chordal \SLEk.

It remains only to prove estimate~\eqref{two}.  This
is trivial if $d \leq 100\, \theta$, so for
the remainder we may assume that $d > 100\, \theta$.
It then follows that $\hat\eta$ does not intersect the line segment $J$ connecting $1$ and $e^{i\theta}$, so $\hat\eta\subseteq\ol\Half\sm(\hat I\cup J)$.
Because the sets $(0,1)$ and $\R\sm[0,1]$ lie in different components of the disconnection $\ol\Half\sm(\hat I\cup J)$,
it is impossible to have one endpoint of $\hat\eta$ in $(0,1)$ and the other outside $[0,1]$.
Since $d>100\,\theta$, this implies
that  both endpoints of $\hat\eta$ lie outside $[0,1]$ (else we would be in the trapped case).
Since $\theta \leq \pi/2$, we can see that 
$d$ is comparable to the radius of the largest disk centered at $1$
that does not intersect $\hat\eta$.

Suppose $\hat\gamma_\tau$ is given and let $B$ be a complex 
Brownian motion starting at $\zeta:= e^{i\theta}$.  Let $U$ be the disk centered at $\zeta$ of radius $d/2$ and
\begin{align*}
T_1 &= \inf\{t: B_t \notin \Half\sm\hat\gamma_\tau\}, \\
T_2 &= \inf\{t: |B_t - \zeta| =  d/2\}=\inf\{t:B_t\notin U\}.
\end{align*}
It is standard that $\Prob\{T_2 < T_1\} = O(\theta/d)$.
Hence it suffices to prove that all  paths with $T_1  \leq T_2$ leave on the 
same side of the domain $(\Half\sm\hat\gamma_\tau,\hat\gamma(\tau),\infty)$, 
in that the images of their endpoints  under the conformal map 
that sends $\Half\sm\hat\gamma_\tau,\hat\gamma(\tau),\infty$ to $\Half,0,\infty$ 
are either all positive or all negative.

\begin{figure}[h]
\centerline{%% Creator: Inkscape inkscape 0.48.4, www.inkscape.org
%% PDF/EPS/PS + LaTeX output extension by Johan Engelen, 2010
%% Accompanies image file 'untrappedcase.pdf' (pdf, eps, ps)
%%
%% To include the image in your LaTeX document, write
%%   \input{<filename>.pdf_tex}
%%  instead of
%%   \includegraphics{<filename>.pdf}
%% To scale the image, write
%%   \def\svgwidth{<desired width>}
%%   \input{<filename>.pdf_tex}
%%  instead of
%%   \includegraphics[width=<desired width>]{<filename>.pdf}
%%
%% Images with a different path to the parent latex file can
%% be accessed with the `import' package (which may need to be
%% installed) using
%%   \usepackage{import}
%% in the preamble, and then including the image with
%%   \import{<path to file>}{<filename>.pdf_tex}
%% Alternatively, one can specify
%%   \graphicspath{{<path to file>/}}
%% 
%% For more information, please see info/svg-inkscape on CTAN:
%%   http://tug.ctan.org/tex-archive/info/svg-inkscape
%%
\begingroup%
  \makeatletter%
  \providecommand\color[2][]{%
    \errmessage{(Inkscape) Color is used for the text in Inkscape, but the package 'color.sty' is not loaded}%
    \renewcommand\color[2][]{}%
  }%
  \providecommand\transparent[1]{%
    \errmessage{(Inkscape) Transparency is used (non-zero) for the text in Inkscape, but the package 'transparent.sty' is not loaded}%
    \renewcommand\transparent[1]{}%
  }%
  \providecommand\rotatebox[2]{#2}%
  \ifx\svgwidth\undefined%
    \setlength{\unitlength}{384.3730957bp}%
    \ifx\svgscale\undefined%
      \relax%
    \else%
      \setlength{\unitlength}{\unitlength * \real{\svgscale}}%
    \fi%
  \else%
    \setlength{\unitlength}{\svgwidth}%
  \fi%
  \global\let\svgwidth\undefined%
  \global\let\svgscale\undefined%
  \makeatother%
  \begin{picture}(1,0.46063899)%
    \put(0,0){\includegraphics[width=\unitlength]{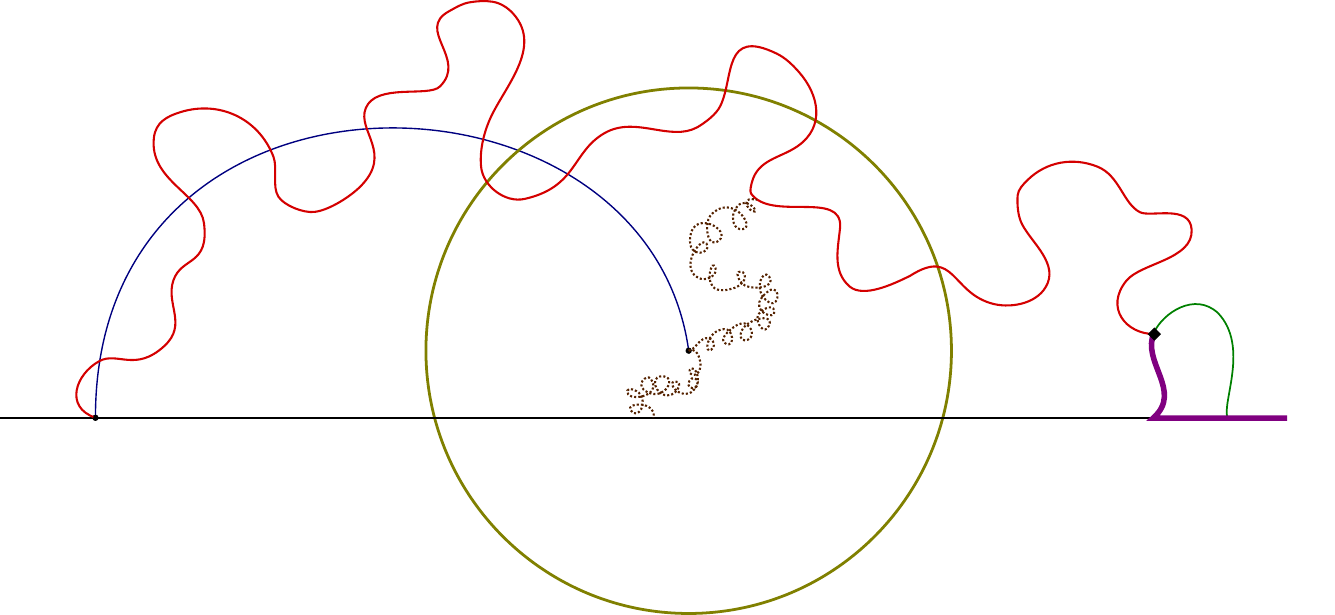}}%
    \put(0.07091753,0.11664288){\color[rgb]{0,0,0}\makebox(0,0)[b]{\smash{$0$}}}%
    \put(0.50349913,0.19531527){\color[rgb]{0,0,0}\makebox(0,0)[rb]{\smash{$\zeta=e^{i\theta}$}}}%
    \put(0.4698523,0.25815336){\color[rgb]{0,0,0}\makebox(0,0)[rb]{\smash{$\hat I$}}}%
    \put(0.89328199,0.24165594){\color[rgb]{0,0,0}\makebox(0,0)[lb]{\smash{$\hat\eta$}}}%
    \put(0.40250371,0.05385484){\color[rgb]{0,0,0}\makebox(0,0)[lb]{\smash{$U$}}}%
    \put(0.8539256,0.18899575){\color[rgb]{0,0,0}\makebox(0,0)[rb]{\smash{$\hat\eta(s)$}}}%
    \put(0.31951721,0.41374394){\color[rgb]{0,0,0}\makebox(0,0)[rb]{\smash{$\hat\gamma$}}}%
    \put(0.96413054,0.10948386){\color[rgb]{0,0,0}\makebox(0,0)[rb]{\smash{$l=I\cup\hat\eta[0,s)$}}}%
    \put(0.52732559,0.16090643){\color[rgb]{0,0,0}\makebox(0,0)[lb]{\smash{$B^{(1)}$}}}%
    \put(0.58481883,0.20596045){\color[rgb]{0,0,0}\makebox(0,0)[lb]{\smash{$B^{(2)}$}}}%
    \put(0.96511771,0.1614321){\color[rgb]{0,0,0}\makebox(0,0)[rb]{\smash{$I$}}}%
  \end{picture}%
\endgroup%
}
\caption{The untrapped case in Proposition~\ref{2sr-finite-time}.}
\label{fig:untrapped}
\end{figure}

To see this, 
suppose that $\hat\gamma(\tau) = \hat \eta(s)$.
If $\hat\eta(0)>1$, let $I=[\hat\eta(0),\infty)$;
otherwise, $\hat\eta(0)<0$ and we let $I=(-\infty,\hat\eta(0)]$.
Then let $l=I\cup\hat\eta[0,s)$, 
which is the trace of a curve in $\ol{\Half\sm\hat\gamma_\tau}$ 
from $\hat\gamma(\tau)$ to $\infty$ that does not intersect~$U$.
Hence, any curve in $\Half\sm\hat\gamma_\tau$
that joins the two sides of $(\Half\sm\hat\gamma_\tau,\hat\gamma(\tau),\infty)$ 
must intersect~$l$ (possibly at an endpoint) and thus must leave $U$.
Therefore, if two Brownian paths $B^{(1)}$ and $B^{(2)}$ starting at $\zeta$ both satisfy $T_1\leq T_2$, 
then they leave on the same side of the domain.
\end{proof}

We may deduce the basic boundary estimate for two-sided radial \SLEk\ run until it hits the target point.
First, we prove a modification of Lemma~2.9 in~\cite{Lcont}.

\begin{lemma}\label{gateway}
There exists $c<\infty$ such that the following holds.
Let $\eta$ be a crosscut of $\D$ that lies outside $\frac12\D$.
Let $\gamma:(0,1]\to\D$ be a curve starting at~$1$ stopped when it first hits~$C_n$, and suppose that $\gamma\cap\eta=\emptyset$.
Let $D$ be the connected component of $\D\sm\gamma$ containing $0$ and $g:D\to\D$ the conformal map with $g(\gamma(1))=1$ and $g(0)=0$.
Then
\[
\frac{\diam(g(\eta))}{\dist(g(\eta),1)}\le
c\,e^{-n/2}\,\diam(\eta).
\]
\end{lemma}

\begin{proof}
Let $H$ denote the connected component of $D \cap \D_1$ containing
the origin.  Note that $\p H \cap C_{1}$ is a countable (perhaps infinite)
collection of open subarcs.  However, there is a unique one that separates
$\eta$ from $0$ in $H$.  Let $l$ denote this subarc. Consider the
conformal rectangle $R$ that is the connected component of $D
\setminus (l \cup C_n)$ that contains both $C_n$ and $l$ on its
boundary.  We can write $\p R  = C_n \cup \p_+ \cup \p_-
 \cup l$, which we consider as arcs of $\p R$ in the sense of prime ends.
 (For example, if $\gamma$ is simple then each point of~$\gamma$
  corresponds  to two points of $\p D$.)
   Note that $g(\p_+)$, $g(\p_-)$ are adjacent intervals
 in $\p \Disk$ with common boundary point $1$,  and $g (\eta)$
 is separated from~$0$ by $g(l)$.  Using conformal invariance
 we can see that 
 \begin{align*}
   \dist(g(\eta),1) &\geq c  \min\{h_D(0,\p_+), h_D(0,\p_-)\}, \\
      \diam g(\eta) &\leq c \, h_D(0,\eta) .
\end{align*}
 There is a conformal transformation 
 \[   f: R \rightarrow \rect_L := \{x+iy: 0 < x < L,\; 0 < y < \pi \} \]
 that maps $\p_-, \p_+$ onto the two horizontal boundaries.  The
 length $L$ is a conformal invariant, but the Beurling estimate
 shows that $L \geq n/2-c$.  If $\zeta \in \rect_L$
 with $\Re\zeta \leq 1$, then, one can check directly that
 \[     h_{\rect_L} (\zeta,L + i(0,\pi))
   \leq c \, e^{-L} \, \min \bigl\{ h_{\rect_L} (\zeta,(0,L)),
       h_{\rect_L} ( \zeta,i\pi +(0,L)) \bigr\}. \]
By starting a Brownian motion from $0$ and 
considering the sequence of hitting times 
of $C_n$ and $f\inv\{\Re\zeta=1\}$ 
before it leaves $H$,
it is not hard to see that
 \begin{align*}
   h_{H }(0,l)
    & \leq  c \, e^{-n/2} \,\min \bigl\{ h_{H }(0,\p_+),
    h_{H }(0,\p_-)\bigr\}\\
     &  \leq   c \, e^{-n/2} \, \min \bigl\{ h_{D}(0,\p_+),
    h_{D}(0,\p_-)\bigr\}\\
    & \leq  c \, e^{-n/2} \, \dist(g(\eta),1).
    \end{align*}
  Finally, a simple application of the Harnack inequality
 implies that if $|\zeta|  \leq e^{-1}$,
 then
 \[ h_{D \setminus \eta}(\zeta,\eta)
   \leq h_{\Disk }(\zeta,\eta)
    \leq c \, h_{\Disk } (0,\eta) \leq c \, \diam (\eta), \]
  and hence
 \begin{align*}
   \diam(g(\eta)) \asymp h_{D \setminus \eta}(0,\eta)
  & \leq  h_{D \setminus \eta}(0,l) \, \sup_{\zeta \in l}
   h_{D \setminus \eta}(\zeta,\eta)\\
& \leq  c \,e^{-n/2} \,  \diam (\eta)\,
  \dist(g(\eta),1).
\qedhere
  \end{align*}
\end{proof}

\begin{corollary}\label{two-sided-radial-full}
If $0<\kappa<8$, there exists $c<\infty$ such that the following holds.
Let $D$ be a simply connected domain with $z,w\in\p D$ and $\zeta \in D$.
Let $\eta$ be a crosscut of $D$
and $\tilde\gamma$ be any curve in $D$ from $z$ to $\zeta$.
Let~$\gamma$ be a two-sided radial \SLEk\ curve in $D$ from $z$ to $w$ through $\zeta$ stopped when it reaches $\zeta$.
Then
\[
\Prob\{\gamma\cap\eta\ne\emptyset\}
\le c\,\exc_D(\tilde\gamma,\eta)^{4a-1}.
\]
\end{corollary}

\begin{proof}
By conformal invariance it suffices to consider the case $D=\D$, $z=1$, $\zeta=0$, $w=e^{-2i\theta}$.
Let $D_t$ be the connected component of $\D\sm\gamma_t$ containing $0$ and let $g_t:D_t\to\D$ be the conformal map sending $0$ to $0$ and $\gamma(t)$ to $1$.
Let $\sigma_n=\inf\{t:|\gamma(t)|=16^{-n}\}$ and $\eta_n=g_{\sigma_n}\circ\eta$.

On the event $\{\gamma_{\sigma_n}\cap \eta = \emptyset\}$,
 Proposition~\ref{2sr-finite-time} implies that 
\[
 \Prob\{\gamma[\sigma_n,\sigma_{n+1}] \cap \eta \ne \emptyset
 \mid \gamma_{\sigma_n}  \}  \leq  
c\,\Bigl(\frac{\diam(\eta_n)}{\dist(1,\eta_n)}\Bigr)^{4a-1},
\]
and  Lemma~\ref{gateway} gives
\[
\frac{\diam(\eta_n)}{\dist(1,\eta_n)}
\le c\, \frac{\diam(\eta)}{\dist(1,\eta)}\,4^{-n}.
\]
Therefore,
\[ \Prob \{\gamma[\sigma_n,\sigma_{n+1}] \mid
 \gamma_{\sigma_n} \cap \eta = \emptyset\} \leq 
     c \,\Bigl(\frac{\diam(\eta)}{\dist(1,\eta)} \, 4^{-n}\Bigr)^{4a-1},\] 
and summing these probabilities, we see that
\[
\Prob\{\gamma\cap\eta\ne\emptyset\}
\le c\, \Bigl(\frac{\diam(\eta)}{\dist(1,\eta)}\Bigr)^{4a-1}.
\qedhere
\]
\end{proof}

The proof of Theorem~\ref{thm3} follows from Corollary~\ref{two-sided-radial-full} by the same argument as for radial \SLEk, which is given in Proposition~\ref{prob-int-Cs}.

%%                                                               %%
%% Use the two commands below for producing your bibliography    %%
%% with bibtex, then comment again the commands and include the  %%
%% content of the .bbl file in this file below the commands.     %%
%%                                                               %%

%\bibliographystyle{amsplain}
%\bibliography{yourbibfilename}

\begin{thebibliography}{9}
%% Total bibitems: 8, successfull: 6, errors: 2
%% Generated through BatchMRef at http://www.e-publications.org

\bibitem{AK}
Tom Alberts and Michael~J. Kozdron.
\newblock Intersection probabilities for a chordal {SLE} path and a semicircle.
\newblock {\em Electron. Commun. Probab.}, 13:448--460, 2008.
\MR{2430712}, \ARXIV{0707.3163}

\bibitem{Lbook}
Gregory~F. Lawler.
\newblock {\em Conformally invariant processes in the plane}, volume 114 of
  {\em Mathematical Surveys and Monographs}.
\newblock American Mathematical Society, Providence, RI, 2005.
\MR{2129588}

\bibitem{Lcont}
Gregory~F. Lawler.
\newblock Continuity of radial and two-sided radial SLE at the terminal
  point.
\newblock In {\em In the tradition of {A}hlfors-{B}ers. {VI}}, volume 590 of
  {\em Contemp. Math.}, pages 101--124. Amer. Math. Soc., Providence, RI, 2013.
\MR{3087930}, \ARXIV{1104.1620}

\bibitem{LL}
Gregory~F. Lawler and Joan~R. Lind.
\newblock Two-sided {SLE}$_{8/3}$ and the infinite self-avoiding
  polygon.
\newblock In {\em Universality and renormalization}, volume~50 of {\em Fields
  Inst. Commun.}, pages 249--280. Amer. Math. Soc., Providence, RI, 2007.
\MR{2310308}

\bibitem {LR} 
Gregory~F.~Lawler\ and\ Mohammad~A.~Rezaei.
\newblock Minkowski content and natural para\-metrization for the Schramm--Loewner evolution. 
\newblock To appear in {\em Ann. Probab.}, \ARXIV{1211.4146}

\bibitem{LW}
Gregory~F.~Lawler and Brent~M.~Werness.
\newblock Multi-point {G}reen's functions for {SLE} and an estimate of
  {B}effara.
\newblock {\em Ann. Probab.}, 41(3A):1513--1555, 2013.
\MR{3098683}, \ARXIV{1011.3551}

\bibitem{MS} Jason Miller\ and\ Scott Sheffield.
\newblock Imaginary geometry IV: interior rays, whole-plane reversibility, and space-filling trees, \ARXIV{1302.4738}

\bibitem{RS}
Steffen Rohde and Oded Schramm.
\newblock Basic properties of {SLE}.
\newblock {\em Ann. of Math. (2)}, 161(2):883--924, 2005.
\MR{2153402}, \ARXIV{math/0106036}

\end{thebibliography}

% add below the content of your .bbl file produced by bibtex.

%%                                                               %%
%% You may add acknowledgments (optional).                       %%
%%                                                               %%

%\ACKNO{}

%%                                                               %%
%% You have reached the end of your document.                    %%
%%                                                               %%

\end{document}